\documentclass{amsart}
\usepackage[]{amsthm} 
\usepackage[]{amssymb} 
\usepackage{amsmath}
\usepackage[hidelinks]{hyperref}
\usepackage{enumitem}
\usepackage{tabularx}
\usepackage{array}
\usepackage{amssymb}
\usepackage{float}
\usepackage{tikz}
\usepackage{arydshln} 
\usepackage{bbm}
\usepackage{cleveref}
\usepackage{mathrsfs,amsmath}
\usepackage{scalerel,stackengine}

\newtheorem{theorem}{Theorem}[section]
\newtheorem*{theorem*}{Theorem}
\newtheorem{prop}{Proposition}
\newtheorem*{prop*}{Proposition}
\newtheorem{conj}{Conjecture}
\newtheorem{lemma}[theorem]{Lemma}

\theoremstyle{definition}

\theoremstyle{remark}

\newtheorem*{remark*}{Remark}

\numberwithin{equation}{section}

\newcommand{\legendre}[2]{\ensuremath{\left( \frac{#1}{#2} \right) }}

\newcommand{\Q}{\mathbb{Q}}
\newcommand{\Z}{\mathbb{Z}}

\renewcommand{\O}{\mathcal{O}}

\begin{document}

\title[Biases for Quadratic Forms]{biases towards the zero residue class for quadratic forms in arithmetic progressions}

\author{Jeremy Schlitt}
\address{Department of mathematics and statistics, Concordia University, 1400 De Maisonneuve Blvd. W., Montr\'eal, Q{c} H3G1M8, CANADA}
\curraddr{D\'epartement de Math\'ematiques et de Statistique, Universit\'e de Montr\'eal, CP 6128 succ. centre, Centre Ville, Montr\'eal, Qc 3HJ3J7, CANADA}
\thanks{The author was supported by the Natural Science and Engineering Research Council via a Canada Graduate Scholarship (CGS-M)}


\date{August 26, 2023 and, in revised form, \today.}

\begin{abstract}
We prove a bias towards the zero residue class in the distribution of the integers represented by binary quadratic forms. In most cases, we prove that the bias comes from a secondary term in an associated asymptotic expansion. This is unlike Chebyshev's bias, which exists somewhere at the level of $O(x^{1/2+\varepsilon}).$ We make a conjecture on the general situation which includes all cases not proven. Results on the distribution of the integers represented by a quadratic form--some of which are of independent interest--are proven along the way. The paper concludes with some numerical data that is illustrative of the bias. 
\end{abstract}

\maketitle

\section{Introduction}
Building upon a well known theorem of Landau \cite{landau1909einteilung}, Bernays \cite{bernays1912ueber}\footnote{See section 2 of \cite{moree2006two} for an excellent exposition on Landau and Bernay's theorem, due to Moree and Osburn. cf. \cite{Jamesquadforms,Pall1943TheDO}.} gave an asymptotic estimate on the number of positive integers not exceeding $x$ that are represented by a given binary quadratic form, henceforth called $B_f(x)$.
\begin{theorem*}[Bernays, 1912]
    $$B_f(x) \sim C(f)\frac{x}{\sqrt{\log x}} \quad (x \to \infty),$$
where $C(f)$ is a constant depending only on $f$.
\end{theorem*}
Above, and below, $f$ is always a positive-definite integral binary quadratic form of negative fundamental discriminant $D$. In the present article, we are concerned with the distribution in arithmetic progressions of integers represented by a given binary quadratic form. Let $B_f(x;q,a)$ be the number of integers not exceeding $x$ that are congruent to $a$ modulo $q$, and represented by the form $f$. As a first result, we have an extension of Bernay's Theorem to arithmetic progressions.

\begin{theorem}
\label{thm1}
Let $q$ be a positive integer such that $(q,2D)=1$. Then for all $a$ satisfying $(a,q) = 1$:
$$B_f(x;q,a) = \frac{1}{q}\prod_{\substack{p|q\\\legendre{D}{p} = -1} } \left(1+\frac{1}{p} \right) B_f(x) + \mathcal{O}\left( \frac{x}{(\log x)^{2/3}} \right).$$
In particular, when $q$ is a prime number, one has
$$B_f(x;q,a) = c(q,a) B_f(x) + \mathcal{O}\left( \frac{x}{(\log x)^{2/3}} \right),$$
where
$$c(q,a) := \begin{cases} \frac{1}{q} \quad & \textit{ if } \legendre{D}{q} = 1\\
\frac{q+1}{q^2} \quad & \textit{ if } \legendre{D}{q} = -1, a \not \equiv 0 \bmod q\\
\frac{1}{q^2} \quad & \textit{ if } \legendre{D}{q} = -1, a \equiv 0 \bmod q
\end{cases}$$
\end{theorem}
\Cref{thm1} tells us what we should expect as the proportion of integers that lie in each non-zero residue class, and when $q$ is prime, this implies the proportion for the zero residue class. However, a look at some of the tables of \Cref{numericaldata} suggests an interesting discrepancy between \Cref{thm1} and actual data. It appears that there is a numerical bias towards the $0$ residue class in all examples. For example, in cases where $\legendre{D}{q} = 1$, we expect a proportion of $1/q$ in each class, and yet we observe that the $0$ class contains more integers than the non-zero classes in an apparent way. We give in \Cref{thm2,thm3} a theoretical explanation for this bias in different cases by computing a secondary term for $B_f(x;q,a)$. Hereinafter, $C(D)$ denotes the group of reduced forms of discriminant $D$ with $h=h(D) := |C(D)|$, and $G(D)$ denotes the genus group, i.e. $G(D) \cong C(D)/C(D)^2$. In the case where $C(D) \cong G(D)$, we explicitly compute the secondary term in the asymptotic expansion of $B_f(x;q,a)$, which accounts for the bias: 
\begin{theorem}
    \label{thm2}
    Let $D$ be a discriminant such that $C(D) \cong G(D)$, and let $f \in C(D)$ be a given form. Let $q$ be a prime modulus for which ${(q,2D) = 1}$. One has
    $$B_f(x;q,a) = c(q,a) \left( a_0 \frac{x}{(\log x)^{1/2}} + a_1\big(1 - \frac{a_0}{a_1} \delta(a,q) \big)\frac{x}{(\log x)^{3/2}}\right) + O\left(\frac{x}{(\log x)^{5/2}}\right) ,$$
    where $c(q,a)$ is as in \Cref{thm1}, $a_0$ and $a_1$ are defined by \eqref{EQ BR L1}, and
    $$\delta(q,a) = \begin{cases}
        \frac{\log q}{2(q-1)}, \quad & \textit{ if } \legendre{D}{q} = 1,\quad a \not \equiv 0 \bmod q\\
        - \frac{\log q}{2}, \quad & \textit{ if } \legendre{D}{q} = 1,\quad a \equiv 0 \bmod q\\
        \frac{\log q}{q-1}, \quad & \textit{ if } \legendre{D}{q} = -1,\quad a \not \equiv 0 \bmod q\\
        - \log q, \quad & \textit{ if } \legendre{D}{q} = -1,\quad a \equiv 0 \bmod q.
    \end{cases}$$
\end{theorem}
 \Cref{thm2} tells us that the secondary term of $B_f(x;q,a)$ is generally larger when $a\equiv 0 \bmod q$, as opposed to when $a \not\equiv 0 \bmod q$, explaining the numerical bias. 

Let us write $1_R(n)$ to be the function that equals $1$ if $n$ is represented by a form in the genus $R \in G(D)$, and $1_R(n)=0$ otherwise. Let us also define
$$B_R(x;q,a) := \sum_{\substack{n \leq x\\n \equiv a \bmod q}}1_R(n).$$ The proofs of \Cref{thm1,thm2} both begin with a computation of $B_R(x;q,a)$. Instrumental in the final step of the proof of \Cref{thm1} is the following strong result of Fomenko (\cite{fomenko_1998}):
\begin{theorem}[Fomenko, 1998]
\label{Fomthm}
Let $f$ be a form in the genus $R \in G(D)$. We have, for all $a,q$,
$$B_f(x;q,a) = B_R(x;q,a) + \O \left( \frac{x}{(\log x)^{2/3}} \right) \quad (x \to \infty).$$
\end{theorem}
 \Cref{Fomthm} tells us that almost all integers represented by \textit{some} form in the genus $R$ are actually represented by \textit{all} forms in that genus. Turning now to the case when there is a single genus of forms, i.e. $G(D) \cong \{1\}$, we prove a theorem for squarefree integers coprime to $2D$, which explains the numerical bias towards the zero residue class modulo $q$ in this setting. Rather astonishingly, though the secondary term in this case is of a different nature than the secondary term of \Cref{thm2}, it still espouses a bias towards the zero residue class.

\begin{theorem}
\label{thm3}
Let $D$ be a discriminant for which $C(D)$ is cyclic of odd order $h$. Let $[f^\star]$ be a generator of $C(D)$, and $H = \langle [f^\star]^{p_0} \rangle$, where $p_0$ is the smallest prime divisor of $h$. Let $q$ be an integer such that $(q,2D) = 1$, and let $B'_f(x;q,a)$ deenote the number of squarefree integers not exceeding $x$ that are coprime to $2D$, are represented by the form $f$, and are congruent to $a \bmod q$. 
Then, one has, for every $a$ such that $(a,q) = 1$,
$$B'_f(x;q,a) \sim A_1(f,q,a) \frac{x}{(\log x)^{1/2}},$$
for some constant $A_1(f,q,a)$, and 
$$B'_f(x;q,a)  =A_1(f,q,a) \frac{x}{(\log x)^{1/2}} - N'_f(x;q,a),$$
where 
$$N'_f(x;q,a) = c'(q,a) A_2(f) \frac{x}{(\log x)^{1-1/(2p_0)}}(\log \log x)^r (1+o(1)),$$
$A_2(f)$ is a positive constant depending only on $f$, 
$$r = \begin{cases} p_0-2, \quad & [f]\in H,\\ p_0-3, \quad & \textit{otherwise,} \end{cases}$$ 
$$c'(q,a) = \frac{1}{\phi(q)}\prod_{p|q}\left(1+\frac{\nu_H(p)}{p} \right)^{-1},$$
and $\nu_H(p)$ is defined to be $1$ if $p$ is represented by a class of forms in $H$, and $\nu_H(p) = 0$ otherwise. In particular, if $q$ is taken to be prime, then
$$c'(q,a) = 
\begin{cases} 
\frac{1}{q+1}, \quad & a \equiv 0 \bmod q \textit{ and }\nu_H(q) = 1\\

\frac{q}{q^2-1}, \quad &\textit{if } a \not\equiv 0 \bmod q \textit{ and }\nu_H (q)= 1 \\

o(1), \quad & a \equiv 0 \bmod q \textit{ and }\nu_H(q) = 0\\

\frac{1}{q-1}, \quad &\textit{if } a \not\equiv 0 \bmod q \textit{ and }\nu_H(q) = 0.
\end{cases}$$
\end{theorem}
Note that $q/(q^2-1) > 1/(q+1)$, so that (for $q$ prime) $-c'(q,a)$ is larger when $a \equiv 0$ and smaller otherwise.

An examination of numerical data suggests that the bias still exists when one removes the restriction that the integers should be squarefree and coprime to $2D$. See \Cref{cn3}. In \Cref{sectremovesf}, we will discuss why dropping the squarefreeness condition makes the constants in our asymptotic expansion inexplicit.

\textbf{Summary of main theorems:} 
\begin{itemize}
    \item If $G(D) \cong C(D)$, then the biased secondary term will be of size $x/(\log x)^{3/2}$, and will arise from a result on the equidistribution of arithmetic progressions for integers represented by a given genus (\Cref{equidgeneralemma}).
    \item If $C(D)$ is cyclic and of odd order, then the biased secondary term will be of size $x (\log \log x)^\beta/(\log x)^\alpha$, for some $\beta \in \Z_{\geq 0}$, and $\alpha \in [5/6,1)$. The secondary term in this case arises from a characterisation of integers which are represented by some form in the genus of $f$ \textit{but not} by $f$ itself, which will generalise the work of \cite{golubeva2001exceptional}.
\end{itemize}

Several recent works have been concerned with some aspects of the behaviour of the form $x^2+y^2$ in arithmetic progressions. In \cite{gorodetsky_2022}, Gorodetsky demonstrates a Chebyshev-type bias towards quadratic residues for numbers of the form $x^2+y^2$. The most salient feature of Gorodetsky's result is that it occurs in a \textit{natural density} sense, unlike the logarithmic density one obtains when examining Chebyshev's bias for primes.

Inspired by the earlier work of Lemke Oliver and Soundararajan \cite{lemkeoliver_soundararajan_2016}, David, Devin, Nam, and the author studied in \cite{david_devin_nam_schlitt_2021} the distribution of consecutive sums of two squares in arithmetic progressions. \cite[Thm 2.4]{david_devin_nam_schlitt_2021} is a special case of \Cref{thm2} in the present article. 

The function field analogue of Landau's theorem has also been explored. In \cite{gorodetsky_rodgers_2021}, Gorodetsky and Rodgers explore the variance of sums of two squares in short intervals for $\mathbb{F}_q[T]$. Theorem B.1 of their article is of interest, as it provides an integral representation for the number of integers that are the sum of two squares not exceeding $x$, with a $O(x^{1/2+\epsilon})$ error term (under GRH). It would be interesting to consider the extension of this theorem to more general families of binary quadratic forms, and perhaps also to $B_f(x,q,a)$.

\specialsection*{ACKNOWLEDGEMENTS}
We thank Professor Chantal David for her support throughout this project, as well as for her careful review of numerous drafts. We would also like to thank the organizers of the MOBIUS ANT seminar for inviting us to present an earlier version of our work in a friendly and constructive environment.

The generation of numerical data in the present paper as well as the numerical approximation of constants have been done using Mathematica \cite{MATH} and Sage \cite{SAGE}. All such computations were done on the author's personal computer, equipped with an AMD Ryzen 7 5800X 8-Core Processor.
%

\section{Proof of Theorems 1.1 and 1.3}\label{MT}
\textit{Proof of Theorem 1.1}: \Cref{Fomthm} tells us that the main term in the asymptotic expansion of $B_R(x;q,a)$ matches that of $B_f(x;q,a)$. In light of this fact, it can be seen that \Cref{thm1} is implied from the following:
\begin{theorem} \label{equidgeneralemma}
    Let $q$ be a positive integer, and let $(q,2D) = 1$, where $D$ is a negative fundamental discriminant. Let $R \in G(D)$ be a genus of forms. Then, for every $J \in \mathbb{N}$, we can write
    \begin{equation}
    \label{EQ BR L1}
        B_R(x) := B_R(x;1,1) = \sum_{j = 0}^J \frac{a_j x}{(\log x)^{1/2 + j}} +O\left(\frac{x}{(\log x)^{3/2+J}}\right)
    \end{equation}
    and for every $a$ satisfying $(a,q) = 1$, we can write
    $$B_R(x;q,a) = \sum_{j = 0}^J \frac{b_j x}{(\log x)^{1/2 + j}} +O\left(\frac{x}{(\log x)^{3/2+J}}\right)$$
    for positive constants $a_j$ and $b_j$ that do not depend on $R$ or on $a$. $a_j$ is defined by \eqref{ai}, and the first two $b_j$ are given by
    \begin{align}
    \label{constant b0 lemma 1}
        b_0 &= a_0\frac{1}{q}\prod_{\substack{ p | q \\ (D|p) = -1}}(1+p^{-1}),\\
        \label{constant b1 lemma 1}
        b_1 &= a_1 \frac{1}{q}\prod_{\substack{ p | q \\ (D|p) = -1}}(1+p^{-1}) \left( 1- \frac{a_0}{2a_1}\left(\sum_{p | q} \frac{\log p}{p-1} - \sum_{\substack{p | q \\ (D|p) = -1}} \frac{\log p}{p+1}\right) \right).
    \end{align}
\end{theorem}

\begin{proof}[Proof of \Cref{equidgeneralemma}]
We loosely follow the proof of a similar theorem which appears in \cite{Lut67}, where an asymptotic for the number of rational integers that are the norm on an algebraic integer of a given quadratic number field is obtained. We begin with a Lemma which completely describes how the genus characters act on the primes.
\begin{lemma}
Let $D$ be the discriminant of the imaginary quadratic field $K$. There is a one-to-one correspondence between the characters of the genus group $G(D)$ and the factorizations $D = uv$ of $D$ into two fundamental discriminants $u$ and $v$ (treating $uv$ and $vu$ as the same factorization). The relation between the genus characters $\psi \in \widehat{G(D)}$ and the factorizations is expressed by
$$\psi(\mathfrak{p}) := \begin{cases} \left(\frac{u}{N(\mathfrak{p})}\right), \textit{ if } (u, \mathfrak{p}) = 1, \\ \vspace{-0.37cm} \\ \left(\frac{v}{N(\mathfrak{p})}\right), \textit{ if } (v, \mathfrak{p}) = 1, \end{cases}$$
where $\mathfrak{p}$ is any prime ideal of $\mathcal{O}_K$.
\end{lemma}
\begin{proof}
 Note that any real character of the ideal class group $C(D)$ will correspond with a character of the genus group $G(D) = C(D)/C(D)^2$, as 
\begin{align*}
    \psi \textit{ is a character of } G(D) & \iff \psi \textit{ is a character of } C(D)\\ & \quad \quad \textit{ satisfying } \psi(g^2 n) = \psi(g'^2 n) \quad \forall g,g',n \in C(D) \\
    & \iff \psi \textit{ is a real character of } C(D).
\end{align*}
The number of genera of discriminant $D$ is $2^{\mu - 1}$. The integer $\mu$ is defined as follows: if $D \equiv 1 \bmod 4$, $\mu = s$, the number of distinct odd prime factors of $D$. If $D = 4n$, then
$$\mu = \begin{cases} s, \quad n \equiv 1 \bmod 4\\
s+1 \quad n \equiv 2,3 \bmod 4\\
s+1 \quad n \equiv 4 \bmod 8\\
s+2 \quad n \equiv 0 \bmod 8
\end{cases}$$
See Theorem 3.15 of \cite{Cox} for the proof of this fact. This proves that there are $2^{\mu - 1}$ characters of $G(D)$. We note that we get one real character $\psi(\mathfrak{p}) \in \widehat{C(D)}$ for each factorization $D = uv$ as defined in the statement of the lemma. One can calculate that there are in total $2^{\mu - 1}$ such factorizations of $D$, so that the set of characters of $G(D)$ which are as in the statement of the lemma (coming from a factorization $D=uv$) are actually all the characters of $G(D)$. This proves the $1-1$ correspondence.
\end{proof}
From each genus character, we construct the generating series
\begin{equation}
\label{gen series b n psi}
    F(s,\psi):=\sum_{n \geq 1} b(n,\psi) n^{-s},
\end{equation}
for $\Re(s) > 1$, where $$b(n,\psi) = \begin{cases} 0, \textit{ if } n \textit{ is not the norm of any ideal,}\\ \psi(\mathfrak{a}), \textit{ if } n = N(\mathfrak{a} ).\end{cases}$$

If $p \in \Z$ is such that $(D|p) = 0$, then $p O_K = \mathfrak{p}^2$, and so $p = N(\mathfrak{p})$, if $(D|p) = 1$, then $p O_K = \mathfrak{p}_1 \mathfrak{p}_2$, so that $p = N(\mathfrak{p}_1) = N(\mathfrak{p}_2)$. Finally, if $(D|p) = -1$, then $p O_K = \mathfrak{p}$, so that $p^2 = N(\mathfrak{p})$, and $p$ is not the norm of any ideal. These remarks imply that $b(n,\psi)$ is well defined. 

Our reason for introducing $b(n,\psi)$ is that it satisfies an orthogonality relation that allows us to detect when $1_R(n) = 1$. Indeed, when $n$ is the norm of some ideal $\mathfrak{a}$, we have:
\begin{equation}
    \label{ortrelationsgenus}
\begin{split}
    1_R(n) &= \frac{1}{G(D)}\sum_{\psi \in \widehat{G(D)}} b(n,\psi) \psi(R)^{-1}\\
    &= \frac{1}{G(D)}\sum_{\psi \in \widehat{G(D)}} \psi(\mathfrak{a})\psi(R)^{-1}\\
    &= \begin{cases}
        1, \quad [\mathfrak{a}] \in R\\
        0, \quad \textit{otherwise.}
    \end{cases}
\end{split}
\end{equation}
Above, $[\mathfrak{a}]$ refers to the ideal class of $\mathfrak{a}$ in the ideal class group $C(D)$. When $n$ is not the norm of any ideal, then $1_R(n) = b(n,\psi) = 0$ for every $\psi \in \widehat{G(D)}$, so the first line of \eqref{ortrelationsgenus} still holds.

Using the multiplicativity of $b(n,\psi)$, we find that 
$$F(s,\psi) = \prod_{p | D} (1-\psi(\mathfrak{p}) p^{-s})^{-1}\prod_{(D|p) = 1} (1-\psi(\mathfrak{p}) p^{-s})^{-1}\prod_{(D|p) = -1} (1-\psi(\mathfrak{p}) p^{-2s})^{-1}.$$
In each term of the above product, $\mathfrak{p}$ denotes any prime above $p$. In the cases when $p O_K = \mathfrak{p}_1\mathfrak{p}_2$, we note that $\psi(\mathfrak{p}_1) = \psi(\mathfrak{p}_2)$, so we can choose either one. We simplify the Euler factors for each prime in the three possible cases (ramifies, splits, inert). 

\textbf{Ramifies:} If $p|D$, then $p|uv$, and $p|u \iff p \nmid v$, and $$(1 - \psi(\mathfrak{p})p^{-s}) = \left(1 - \left(\frac{u}{p}\right)p^{-s}\right)\left(1 - \left(\frac{v}{p}\right)p^{-s}\right),$$
since one of the two factors on the right hand side is always $1$. 

\textbf{Splits:} If $(D|p) = 1$, then $\left(\frac{u}{p}\right)\left(\frac{v}{p}\right) = 1$, and 
\begin{align*}
    (1 - \psi(\mathfrak{p})p^{-s}) &= \left(1 - \left(\frac{u}{p}\right)p^{-s}\right)\\ &= \left(1 - \left(\frac{v}{p}\right)p^{-s}\right)\\ &= \left(1 - \left(\frac{u}{p}\right)p^{-s}\right)^{1/2}  \left(1 - \left(\frac{v}{p}\right)p^{-s}\right)^{1/2}.
\end{align*}
\textbf{Inert:} if $(D|p) = -1$, then $N(\mathfrak{p}) = p^2$, hence $\left(\frac{u}{N(\mathfrak{p})}\right) = \left(\frac{v}{N(\mathfrak{p})}\right) = 1$, and
$$(1 - \psi(\mathfrak{p})p^{-2s}) = (1 -p^{-2s}) = \left(1 - \left(\frac{u}{p}\right)p^{-s}\right)\left(1 - \left(\frac{v}{p}\right)p^{-s}\right).$$
Putting all this information together, we find that
\begin{align*}
    F(s,\psi) &= \prod_{p | D} \left(1 - \left(\frac{u}{p}\right)p^{-s}\right)^{-1} \left(1 - \left(\frac{v}{p}\right)p^{-s}\right)^{-1} \\ &\times \prod_{(D|p) = 1}\left(1 - \left(\frac{u}{p}\right)p^{-s}\right)^{-1/2}  \left(1 - \left(\frac{v}{p}\right)p^{-s}\right)^{-1/2}\\ &\times \prod_{(D|p) = -1} \left(1 - \left(\frac{u}{p}\right)p^{-s}\right)^{-1}\left(1 - \left(\frac{v}{p}\right)p^{-s}\right)^{-1}.
\end{align*}
Let us define
\begin{align*}
    L_u(s) &= \prod_p \left(1-\left(\frac{u}{p}\right)p^{-s}\right)^{-1}\\
    L_v(s) &= \prod_p \left(1-\left(\frac{v}{p}\right)p^{-s}\right)^{-1}.
\end{align*}
We now have
\begin{equation}
    \label{eqgenserieslemma1}
    F(s,\psi)^2 = L_u(s)L_v(s) A(s),
\end{equation}
where $$A(s) = \prod_{p|D} \left(1 - \left(\frac{u}{p}\right)p^{-s}\right)^{-1} \left(1 - \left(\frac{v}{p}\right)p^{-s}\right)^{-1} \prod_{(D|p) = -1} (1- p^{-2s})^{-1}.$$
We note that $A(s)$ is analytic for $\Re(s) > 1/2$, and $L_u(s)L_v(s)$ is entire unless $u = 1, v=D$, since when $u \ne 1, v \ne D$, both $L_u$ and $L_v$ can be viewed as Dirichlet L-functions of non-principal characters.\footnote{$D$ has no square factors, except maybe $4$.}~We note that the pair $u=1,v=D$ corresponds to the principal genus character $\psi_0$ and we have 
$$L_1(s)L_D(s) = \zeta(s)L_D(s).$$
Taking (principal) square roots on either side of \eqref{eqgenserieslemma1}, we have
$$F(s,\psi_0) = \left(\zeta(s)A(s)L_D(s)\right)^{1/2}.$$
We wish to use the Dirichlet series $F(s,\psi_0)$ to make a conclusion about the partial sums $\sum_{n \leq x} b(n,\psi_0)$. We can make use of Perron's formula in the usual way to get the relation
$$\sum_{n \leq x} b(n,\psi_0) = \frac{1}{2 \pi i} \int_{2 - i \infty}^{2 + i \infty} F(s,\psi_0) \frac{x^s}{s}ds,$$
but we are not able to apply Cauchy's theorem to evaluate the contribution from the singularity at $s=1$ of the integrand, since it is not a pole. Instead, we appeal to the Landau--Selberg--Delange method:
\begin{theorem} \cite[Theorem $13.2$]{koukoulopoulos_2019} \label{LSD} Let $f(n)$ be a multiplicative function with generating series $F(s) = \sum_{n \geq 1} f(n) n^{-s}$. Suppose there exists $\kappa \in \mathbb{C}$ such that for $x$ large enough
	$$\sum_{p \leq x} f(p) \log{p} = \kappa x + O_A \left( x/(\log{x})^A \right),$$
	{for each fixed $A>0$,} and such that $|f(n)| \leq \tau_k(n)$ for some $k \in \mathbb{N}$, where $\tau_k$ is the $k$-th divisor function.
	For $j \geq 0$, let $\widetilde{c_j}$  be the Taylor coefficients about 1 of the function $(s-1)^\kappa F(s)/s$. Then, for any $J  \in \mathbb{N}$, and $x$ large enough, we have
		$$
	\sum_{n \leq x} f(n) = x \sum_{j=0}^{J} \widetilde{c_j} \frac{(\log{x})^{\kappa-j-1}}{\Gamma(\kappa-j)} + O \left( \frac{x}{(\log{x})^{J+2-\Re(\kappa)}}\right).
	$$
	\end{theorem}

Applying \Cref{LSD} to $F(s,\psi)$, we conclude that for any $J \geq 0$, one has
\begin{equation}
\label{intermed step l1 eq}
    \sum_{n \leq x}b(n,\psi) = \delta_{\psi} \sum_{j=0}^J\frac{\Tilde{a_j} x}{\Gamma(1/2-j) (\log x)^{1/2+j}} + \mathcal{O}\left(\frac{x}{ (\log x)^{J+3/2}}\right),
\end{equation}
where $\delta_{\psi} = 0$, unless $\psi= \psi_0$, in which case $\delta_{\psi_0} = 1$, and $\Tilde{a_j}$ is the $j^{th}$ Taylor coefficient of $(s-1)^{1/2}F(s,\psi_0)/s$. This is essentially a consequence of the Prime Number Theorem in arithmetic progressions. When $\psi \ne \psi_0$, then
\begin{align*}
    \sum_{p \leq x}b(p,\psi)\log p &\ll \sum_{\substack{p \leq x\\\legendre{D}{p}} = 1} \legendre{u}{p}\log p\\
    &\ll_A \frac{x}{(\log x)^A},
\end{align*}
for any $A>0$, since $u \ne 1$ which makes $\legendre{u}{\cdot}$ a non-principal Dirichlet character.

Combining \eqref{ortrelationsgenus} with \eqref{intermed step l1 eq} we find that for each fixed genus $R \in G(D)$ we have, for any $J \geq 0$,
\begin{equation}
    \begin{split}
        B_R(x) &:= \sum_{\substack{n \leq x}} 1_R(n)\\&= \frac{1}{|G(D)|}\sum_{n \leq x} \sum_{\psi \in \widehat{G(D)}} b(n,\psi) \psi(R)^{-1} \\
        &= \frac{1}{|G(D)|} \sum_{\psi \in \widehat{G(D)}}\psi(R)^{-1} \delta_{\psi} \sum_{j=0}^J\frac{\Tilde{a_j} x}{\Gamma(1/2-j) (\log x)^{1/2+j}} + \mathcal{O}\left(\frac{x}{ (\log x)^{J+3/2}}\right)\\ 
        &= \frac{1}{|G(D)|} \sum_{j=0}^J\frac{\Tilde{a_j} x}{\Gamma(1/2-j) (\log x)^{1/2+j}} + \mathcal{O}\left(\frac{x}{ (\log x)^{J+3/2}}\right).
    \end{split}
\end{equation}
We then adopt the notation
\begin{equation}
    \label{ai}
    a_j = \frac{\Tilde{a}_j}{\Gamma(1/2-j)|G(D)|}.
\end{equation}
We can repeat this argument with the addition of a congruence condition modulo $q$ inserted. Define
$$F(s,\psi,\chi) := \sum_{n \geq 1} b(n,\psi) \chi(n) n^{-s},$$
where $\chi$ is a Dirichlet character modulo $q$. Applying \eqref{ortrelationsgenus} as well as the orthogonality relations for Dirichlet characters, one has
\begin{equation}
\label{ortrel psi and chi}
    \frac{1}{\phi(q) |G(D)|}\sum_{\chi \bmod q} \sum_{\psi \in \widehat{G(D)}} b(n,\psi) \chi(n) \chi^{-1}(a) \psi^{-1}(R) = \begin{cases}
        1, \quad 1_R(n) = 1 \textit{ and }n \equiv a \bmod q\\
        0, \quad \textit{otherwise.}
    \end{cases}
\end{equation}

Since we are assuming $(q,D) = 1$, then the conductor of $\chi$ will be coprime to $D$. As such, the only choice of characters that makes $F(s,\psi,\chi)$ have a singularity at $s=1$ is when $\psi = \psi_0$, and $\chi=\chi_0$, the principal character modulo $q$. We then have

\begin{equation}
\label{relationchiwithnonchilemma1}
    \begin{split}
    F(s,\psi_0,\chi_0) &= C_q(s)\left(\zeta(s)A(s)L_D(s)\right)^{1/2}\\
    &=C_q(s)F(s,\psi_0),
\end{split}
\end{equation}
where
$$C_q(s) := \prod_{p|q}(1-p^{-s})\prod_{\substack{p|q \\ (D|p) = -1}}(1+p^{-s}).$$
Applying \Cref{LSD} again, this time to each $F(s,\psi,\chi)$, we conclude that for any $J \geq 0$, one has
\begin{equation}
\label{intermed step 2 l1 eq}
    \sum_{n \leq x}b(n,\psi)\chi(n) = \delta_{\psi,\chi} \sum_{j=0}^J\frac{\Tilde{b_j} x}{\Gamma(1/2-j) (\log x)^{1/2+j}} + \mathcal{O}\left(\frac{x}{ (\log x)^{J+3/2}}\right),
\end{equation}
where $\delta_{\psi,\chi} = 0$, unless $\psi= \psi_0$ and $\chi = \chi_0$, in which case $\delta_{\psi_0,\chi_0} = 1$. Above, $\Tilde{b_j}$ is the $j^{th}$ Taylor coefficient of $(s-1)^{1/2}F(s,\psi_0,\chi_0)/s$. Combining \eqref{ortrel psi and chi} with \eqref{intermed step 2 l1 eq}, then for every $a$ such that $(a,q) = 1$, one has
$$\sum_{\substack{n \leq x\\ n \equiv a \bmod q }} 1_R(n) = \frac{1}{\phi(q)}\frac{1}{|G(D)|}  \sum_{j=0}^J\frac{\Tilde{b_j} x}{\Gamma(1/2-j) (\log x)^{1/2+j}} + \mathcal{O}\left(\frac{x}{ (\log x)^{J+3/2}}\right).$$
This proves \Cref{equidgeneralemma}, as we see that 
$$b_j = \frac{\Tilde{b_j}}{\Gamma(1/2 - j) \phi(q)|G(D)|},$$
and the constants $b_j$ are independent of $R$. We can easily express the $b_j$ in terms of the $a_j$ as in the statement of the lemma by comparing the Taylor coefficients of $F(s,\psi_0)$ and $F(s,\psi_0,\chi_0)$; these generating series differ only by the factor $C_q(s)$, as can be seen in \eqref{relationchiwithnonchilemma1}.
\end{proof}

This also concludes the proof of \Cref{thm1} by our remarks at the start of this section. Additionally, \Cref{equidgeneralemma} immediately implies \Cref{thm2}.
\begin{proof}[Proof of \Cref{thm2}]
   We begin by noting that when there is a single form per genus, then $B_f(x;q,a) = B_R(x;q,a)$. As in \eqref{constant b1 lemma 1}, the constant multiplying the secondary term for $B_f(x;q,a)$ will be
$$\frac{C_q(1)}{\phi(q) |G(D)|}a_1\left( 1-\frac{a_0}{a_1}\frac{C'_q(1)}{2C_q(1)} \right) $$
when $a \not \equiv 0 \bmod q$, and 
$$\frac{1-C_q(1)}{|G(D)|}a_1 \left(1 + \frac{a_0}{a_1}\frac{C_q'(1)}{2(1-C_q(1))}\right)$$
when $a \equiv 0 \bmod q$ (again, assuming $q$ is prime). Simplifying the above and noting that $a_0/a_1 > 0$ immediately yields \Cref{thm2}. 
\end{proof}
As such, \Cref{equidgeneralemma} gives us an explanation for the numerical bias towards the zero residue class.
\section{Exceptional Integers in Arithmetic Progressions}
\label{sect2}
In this section, $q$ denotes a prime modulus, and we assume that $C(D)$ is cyclic of odd order $h$. We will show how the bias towards the zero residue class\footnote{At least, for squarefree integers that are coprime to $2D$. See \cref{sectremovesf} for details on the (still unproven) general case.} in this case arises from computing the secondary term in the asymptotic expansion of $B'_f(x;q,a)$.

When there is more than one form in the given genus $R$, we do not have the equality $B_f(x;q,a) = B_R(x;q,a)$, except in the sense of \Cref{Fomthm}. The secondary term of $B'_f(x;q,a)$ cannot be computed using \Cref{Fomthm}, due to the $O(x/(\log x)^{2/3})$ error term. We must, however, use a version of \Cref{Fomthm} to compute the main term of $B'_f(x;q,a)$.

\subsection{Proof of \Cref{thm3}}
\label{pf gol thm ap sect}
 Due to \Cref{Fomthm}, studying the secondary term of $B_f(x;q,a)$ is equivalent to estimating the number of integers which are represented by a form from the genus containing $f$, but not represented by $f$ itself. We refer to these as the ``exceptional integers" for $f$. We define $$N_f(x,q,a) := \#\{n \leq x: n \equiv a \bmod q, n \text{ is exceptional for }f\}.$$ It is clear that one has
\begin{equation}
\label{subtractexep}
    B_f(x;q,a) = B_R(x;q,a) - N_f(x,q,a).
\end{equation}
\subsubsection{Summary of Golubeva's Paper}
In \Cref{thm2}, $N_f(x,q,a)$ was always $0$. \Cref{Fomthm} can be viewed as a result giving an upper bound on the size of $N_f(x,q,a)$, and a similar result appears in \cite{golubeva1996representation}. Computing asymptotics for these exceptional integers turns out to be tricky, with the only result known to the author being the following of Golubeva \cite{golubeva2001exceptional}:
\begin{theorem}[Golubeva, 2001]\label{goluthm}
Let $C(D)$ be cyclic of odd order $h$. Let $[f^\star]$ be a generator of $C(D)$. Let $p_0$ be the smallest prime divisor of $h$, and let $H \subset C(D)$ be the subgroup generated by $[f^\star]^{p_0}$. Let $N_f(x):=N_f(x,1,1)$. Then (writing $[f]$ to be the class that $f$ lies in), 

if $[f] \in H$, one has
$$N_f(x) = C_1 \frac{x}{(\log x)^{1-1/(2p_0)}}(\log \log x)^{p_0-2}(1+o(1)),$$
and if $[f] \not \in H$, then one has
$$N_f(x) = C_2 \frac{x}{(\log x)^{1-1/(2p_0)}}(\log \log x)^{p_0-3}(1+o(1)).$$
where $C_1,C_2$ are positive constants depending only on $f$.
\end{theorem}
The constants $C_1$ and $C_2$ are not computed explicitly in Golubeva's paper. See \Cref{sectremovesf} for more comments on the computation of these constants, which turn out to be difficult to do explicitly. Golubeva proves \Cref{goluthm} through a series of lemmas. Lemmas 2-6 of their paper imply that associated to each class of forms $[f] \in C(D)$ is a finite set of tuples $([f_{1}],[f_{2}],\cdots [f_{r}]), \quad [f_i] \in C(D)$, such that once one discards a sparse set, all squarefree integers $n$ which are exceptional for $f$ and coprime to $2D$ will be of the form
\begin{equation}
\label{formgol}
    n = m p_{1} \cdots p_{r},
\end{equation}
where each prime divisor of $m$ is represented by a form in $H$, and $p_{i}$ is any prime represented by $[f_{i}]$. One also has that
$$r = \begin{cases}
    p_0-2, [f] \in H\\
    p_0-3, [f] \not\in H.
\end{cases}$$
\begin{remark*}
    We will consider two tuples $([f_1],\cdots, [f_r])$ and $([g_1],\cdots, [g_r])$ to be distinct if there is no set of indices $i,j$ such that $[f_i] = [g_j]^{\pm 1}$ for $1 \leq i,j \leq r$. If $p_i$ is any prime represented by $[f_i]$, and $q_i$ is any prime represented by $[g_i]$, then $p_1 \cdots p_r$ may be equal to $q_1 \cdots q_r$ if and only if the tuples $([f_1],\cdots, [f_r])$ and $([g_1],\cdots, [g_r])$ are \textbf{not} disctinct. This follows from the fact that a prime is uniquely represented by a class of forms and that class' inverse. 
\end{remark*}
There is a more explicit description of these tuples in both cases. Again taking $[f^\star]$ to be a generator of $C(D)$, we may write $[f_{i}] = [f^\star]^{e_i}$ for each $i$. Lemmas 5 and 6 of Golubeva's paper imply the following two propositions:
\begin{prop}
\label[prop]{case1}
    Let $[f] \in H$ and let $p_0$ be defined as above. The squarefree integers coprime to $2D$ which are exceptional for $f$ (after discarding a sparse set) are those of the form
    $$mp_1p_2\cdots p_{p_0-2},$$
    where each prime divisor of $m$ is represented by a form in $H$, and $p_i$ is represented by the class of forms $[f_i] = [f^\star]^{e_i}$. The set of choices for $(e_1,e_2,\cdots,e_{p_0-2})$ is the set of diagonal nonzero tuples $(a,a,\cdots,a), \; a \in (\Z/p_0\Z)^\ast$.
\end{prop}
    \begin{prop}
    \label[prop]{case2}
        Let $[f] \not\in H$ and let $p_0$ be defined as above. Write $[f] = [f^\star]^{e^\star}$. The squarefree integers coprime to $2D$ which are exceptional for $f$ (after discarding a sparse set) are those of the form
    $$mp_1p_2\cdots p_{p_0-3},$$
    where each prime divisor of $m$ is represented by a form in $H$, and $p_i$ is represented by the class of forms $[f_i] = [f^\star]^{e_i}$. The set of choices for $(e_1,e_2,\cdots,e_{p_0-3})$ being the set of tuples $(a,a,\cdots,a),\; a \in (\Z/p_0\Z)^\ast$, with the additional condition $a \not \equiv e^\star \bmod p_0$.
    \end{prop}

\Cref{goluthm} then follows from a lemma on counting integers of the form \eqref{formgol} for each individual tuple $([f_1], \cdots, [f_r])$ which is suitable, and then summing the results over all the distinct tuples:

\begin{lemma}
\label{gollemma1}
    Let $C(D)$ be cyclic of odd order $h$. Let $H \subset C(D)$ be a proper subgroup. Fix $[f_1],\dots, [f_r]$, (not necessarily distinct) classes of forms not belonging to $H$. We wish to count integers of the form 
        \begin{equation}
        \label{form gol alt}
        \begin{split}
            &n = mp_1\cdots p_r,\\ &p|m \implies  p \textit{ represented by a form }f \in H,\\
            &p_j \textit{ represented by the class } [f_j].
        \end{split}
        \end{equation}
     Let us define $S_{[f_1],\cdots [f_r]}(x,1,1) = S(x)$ by
    \begin{align*}
        S(x) := \#\{n \leq x: n \textit{ squarefree, }(n,2D)=1, n \textit{ satisfying }\eqref{form gol alt}\}.
    \end{align*}
    Then we have, as $x \to \infty$,
    \begin{equation}
    \label{golulemma1eq}
       S(x) = A(f_1,\cdots,f_r) \frac{x}{(\log x)^{1-|H|/(2h)}}(\log \log x)^r(1+o(1)),
    \end{equation}
    where $A(f_1,\cdots,f_r)$ is a constant depending only on the tuple $[f_1],\cdots, [f_r]$.
\end{lemma}
As stated in \Cref{case1} and \Cref{case2}, there is a finite set of tuples $(f_1,\cdots,f_r)$ for which an integer $n$ of the form of \eqref{form gol alt} can be exceptional for $f$. Summing \eqref{golulemma1eq} over all such \textit{distinct}\footnote{In the sense of the remark preceding \Cref{case1} amd \Cref{case2}.} tuples yields \Cref{goluthm} for squarefree integers coprime to $2D$. The general theorem follows by noting that the inclusion of all integers only changes the asymptotic by a constant factor. See \Cref{sectremovesf} for a comment on the removal of the ``squarefree'' condition.

 \subsubsection{Modified Golubeva Lemma and its Proof}
 We consider now the behaviour of the exceptional integers in arithmetic progressions, and we obtain the following generalization of \Cref{gollemma1}: 

\begin{lemma}
\label{gollemma1ap}
   Let $q$ be a chosen modulus, $(q,2D) = 1$. Let $H$, $[f_1],\cdots,[f_r]$ be as in \Cref{gollemma1}. Let us define $S_{[f_1],\cdots [f_r]}(x,q,a) = S(x,q,a)$ by $$S(x,q,a) := \#\{n \leq x: n \textit{ squarefree, }(n,2D)=1, n \equiv a \bmod q, n \textit{ satisfying }\eqref{form gol alt}\}.$$ Then, for $(a,q) = 1$, one has, as $x \to \infty$,

    \begin{equation}
    \label{golulemmaapeq}
       S(x,q,a)= \frac{1}{\phi(q)}\prod_{p|q}\left(1+\frac{\nu_H(p)}{p} \right)^{-1} S(x)(1+o(1)),
    \end{equation}
    where $$\nu_H(n) := \begin{cases} 1, n \textit{ squarefree, }(n,2D) = 1, p|n \implies  p \textit{ represented by a form } f \in H,\\
    0, \textit{ otherwise.}\end{cases}$$
\end{lemma}
To prove \eqref{golulemmaapeq}, we will need  to know about the distribution of \textit{primes }represented by a given quadratic form in an arithmetic progression. The needed information is contained in the following lemma.
\begin{lemma}
    \label{lemma CDT ap}
    Let $q$ be a chosen modulus, $(q,2D) = 1$. Fix a binary quadratic form $f$ of discriminant $D$. We have
    \begin{equation}
    \label{lemma CDT ap eq}
        \sum_{\substack{p \leq x\\ p \equiv a \bmod q}}1_f(p) =\frac{1}{\phi(q)} \frac{\delta(f)}{h} li(x) +O(x\exp{(- c \sqrt{\log x})}),
    \end{equation}
    for some positive constant $c$, where $1_f(n) = 1$ if $n$ is represented by $f$ and $1_f(n) = 0$ otherwise, and
    $$\delta(f) = \begin{cases} \frac{1}{2}, \quad & \textit{ if the class containing } f \textit{ has order } \leq 2 \textit{ in } C(D)\\
    1, & \textit{ otherwise.}\end{cases}$$
\end{lemma}
\begin{remark*}
    It is important to recall that $1_f(p) = 1_{f'}(p)$ for all $f' \in [f]$, and that a prime number $p$ is uniquely represented by a class of forms $[f]$, and the inverse of that class of forms, $[f]^{-1}$. 
\end{remark*}
\begin{proof}[Proof of \Cref{lemma CDT ap}]
    Let $K$ be the imaginary quadratic field of discriminant $D$, and let $H$ be the Hilbert class field of $K$, so that $Gal(H/K) \cong C(D)$. For a given class $\mathcal{C} \in C(D)$, we wish to find an asymptotic for the size of the set
    $$S_{\mathcal{C},a}:= \{\mathfrak{p} \in \mathcal{O}_K : N(\mathfrak{p})\leq x, N(\mathfrak{p}) \equiv a \bmod q, \mathfrak{p} \in \mathcal{C}\}.$$
    The prime ideals in the set $S_{\mathcal{C},a}$ satisfy two Chebatorev conditions simultaneously, and we may apply the Chebatorev density theorem to count them. First, let $\zeta_q = \zeta = e^{2 \pi i /q}$, and consider the following lattice of extensions:
    \begin{center}
        \begin{tikzpicture}
    \node (Q0) at (0,-2) {$\mathbb{Q}$};
    \node (Q1) at (0,0) {$K$};
    \node (Q2) at (2,2) {$K[\zeta]$};
    \node (Q3) at (0,4) {$H[\zeta]$};
    \node (Q4) at (-2,2) {$H$};

    \draw (Q0)--(Q1) node [pos=0.7, right,inner sep=0.25cm] {2};
    \draw (Q1)--(Q2) node [pos=0.7, below,inner sep=0.25cm] {$\phi(q)$};
    \draw (Q1)--(Q4) node [pos=0.7, below,inner sep=0.25cm] {$h$};
    \draw (Q3)--(Q4);
    \draw (Q2)--(Q3);
    \draw (Q1)--(Q3);
    \end{tikzpicture}
    \end{center}
    Note that as long as $h > 1$, then $Gal(K[\zeta]/K) \cong (\Z/q\Z)^*$, since $\zeta^b \not \in K$ for any integer $b$ satisfying $(b,q) = 1$ as $K$ does not contain any $q^{th}$ roots of unity other than $\pm1$.\footnote{In fact, the claim holds as long as $K \neq \Q(\sqrt{m})$ for $m \in \{-1,-3\}$. Either way, we have already dealt with the $h=1$ case in Theorem 1.} Since $H$ is the maximal unramified abelian extension of $K$, and $K[\zeta]/K$ is a finite Galois extension where $q$ is completely ramified, then $H \cap K[\zeta] = K$, and as such we have an isomorphism
    \begin{align}
    \label{gal iso lemm CDT AP}
        \begin{split}
            Gal(H[\zeta]/K) &\cong Gal(H/K)\times Gal(K[\zeta]/K)\\
        \sigma &\to (\sigma|_H,\sigma|_{K[\zeta]})\\
        Art(\mathfrak{p},H[\zeta]/K) &\to (Art(\mathfrak{p},H/K),Art(\mathfrak{p},K[\zeta]/K)),
        \end{split}
    \end{align}
    where $Art$ denotes the Artin symbol. Applying an effective Chebatorev density theorem to the extension $H[\zeta]/K$ tells us that for any given $\sigma \in Gal(H[\zeta]/K)$, one has
    $$|\{\mathfrak{p} \in \mathcal{O}_K : N(\mathfrak{p})\leq x, Art(\mathfrak{p},H[\zeta]/K) = \sigma\}| = \frac{1}{h \phi(q)} Li(x) + O(x\exp{(- c \sqrt{\log x})}).$$
    Note that the ramified primes contribute only $O(1)$ to our count, and are included in the error term. On the other hand, by the isomorphism of \eqref{gal iso lemm CDT AP}, for any given $(\mathcal{C},a) \in C(D) \times (\Z/q\Z)^* \cong Gal(H/K) \times Gal(K[\zeta]/K)$, we can find a $\sigma \in Gal(H[\zeta]/K)$ which is mapped to this pair. With this choice of $\sigma$ in the above equation, one gets
     \begin{align*}
         |S_{\mathcal{C},a}| &=|\{\mathfrak{p} \in \mathcal{O}_K : N(\mathfrak{p})\leq x, Art(\mathfrak{p},H/K) = \mathcal{C}, Art(\mathfrak{p},K[\zeta]/K) = a\}|\\ &= \frac{1}{h \phi(q)} Li(x) + O(x\exp{(- c \sqrt{\log x})}).
     \end{align*}
     Above, $c$ is some positive constant which can be computed. In fact, in our case one can take $c =(99 * \sqrt{h *\phi(q)})^{-1}$ (see, for example, \cite[Thm 1.1]{winckler2013th},  for details on the computation of such a constant).
     To finish the proof, define the sets
     $$S'_{\mathcal{C},a}=\{\mathfrak{p} \in \mathcal{O}_K: \mathfrak{p}\overline{\mathfrak{p}} = p, p \equiv a \bmod q,\:\mathfrak{p} \in \mathcal{C}\},$$
     $$S''_{\mathcal{C},a}=\{p \in \Q, p \equiv a \bmod q,\:p=N(\mathfrak{p}),\:\mathfrak{p} \in \mathcal{C}\}.$$
     We note that
     $$|S'_{\mathcal{C},a}| = |S_{\mathcal{C},a}| + O(x^{1/2}),$$
     $$|S''_{\mathcal{C},a}|= \sum_{\substack{p \leq x\\p \equiv a \bmod q}}1_f(n) + O(x^{1/2}), \textit{ for any } f \in \mathcal{C}.$$
      The map $\mathfrak{p} \to N(\mathfrak{p})$ induces a correspondence between $S'_{\mathcal{C},a}$ and $S''_{\mathcal{C},a}$ that is two-to-one if $\mathcal{C} = \mathcal{C}^{-1}$ in $C(D)$. Otherwise, the map is one-to-one. This introduces the factor $\delta(f)$ into our final count for the size of $S''_{\mathcal{C},a}$.
\end{proof}
\begin{proof}[Proof of \Cref{gollemma1ap}]
     The function $\nu_{H}$ is multiplicative by definition. We claim that for $(q,2D)$ = 1, one has
    \begin{equation}
    \label{nu fn ap eq}
        \sum_{\substack{n \leq x\\n \equiv a \bmod q}} \nu_H(n) = \frac{1}{\phi(q)}\sum_{\substack{n \leq x \\ (n,q) = 1}} \nu_H(n)  + O\left(\frac{x}{(\log x)^A}\right)
    \end{equation}
    for all $A>0$, i.e.~only the principal character mod $q$ will make a contribution to the main term. Indeed, for $\chi$ a non-principal character modulo $q$, one has for $\Re(s) >1/2$ that
    \begin{align}
        \sum_{n \geq 1}\frac{\nu_H(n) \chi(n)}{n^s} &= \prod_{p} \left(1+\frac{\nu_H(p)\chi(p)}{p^s}\right) \notag\\
        &= \exp \left(\sum_{p}\log \left(1+\frac{\nu_H(p)\chi(p)}{p^s}\right)\right) \notag\\
        &\ll \exp \left(\sum_{p} \frac{\nu_H(p)\chi(p)}{p^s}\right). \label{eqvhvhi}
    \end{align}
    By applying \Cref{lemma CDT ap}, summing \eqref{lemma CDT ap eq} over all classes of forms in $H$, we see that 
    \begin{align*}
        \sum_{p \leq y} \nu_H(p) \chi(p) &= \sum_{(a,q)=1}\chi(a)\left(\sum_{\substack{p \leq y \\ p \equiv a \bmod q}} \nu_{H}(p)\right)\\
        &= \sum_{(a,q)=1}\chi(a)\left(\sum_{\substack{p \leq y \\ p \equiv a \bmod q}} 1_{f_0}(p) + \frac{1}{2}\sum_{f \in H \setminus [f_0]}\sum_{\substack{p \leq y \\ p \equiv a \bmod q}} 1_{f}(p)\right)\\
        &= \sum_{(a,q)=1}\chi(a)\left(\frac{1}{2 \phi(q) h}Li(y) + \frac{|H|-1}{2 \phi(q) h}Li(y) + O_a(y\exp{(- c \sqrt{\log y})})\right)\\ 
        &= \frac{|H|}{2 \phi(q) h}\sum_{(a,q)=1}\chi(a)\left(Li(y) + O_a(y\exp{(- c \sqrt{\log y})})\right)\\ 
        &= O(y\exp{(- c \sqrt{\log y})}).
    \end{align*}
    We introduced a factor $1/2$ in the second line above to avoid double-counting; each prime is simultaneously represented by the class $[f]$ and the class $[f]^{-1}$. Since $|C(D)|$ is odd, $[f] = [f]^{-1}$ if and only if $[f] =[f_0]$, and so there is no double counting for $f_0$. Hence, by partial summation, we have
    \begin{align*}
        \sum_{p \leq y} \frac{\nu_H(p)\chi(p)}{p^s} &= y^{-s} \sum_{p \leq y}\nu_H(p)\chi(p) + s\int_2^y t^{-s-1}\sum_{p \leq t}\nu_H(p)\chi(p) \:dt \\
        &\ll sy^{1-s}\exp{(-c \sqrt{\log y})}
    \end{align*}
    We see that the limit as $y$ tends to infinity converges if, say ${\Re(s) \geq 1-1/(\log y)^{1/2}}$. In this range, using \eqref{eqvhvhi} we have
    $$\sum_{n \geq 1}\frac{\nu_H(n) \chi(n)}{n^s} \ll 1.$$
    Using partial summation again, with ${\Re(s) = 1-1/(\log y)^{1/2}}$, we have
    \begin{align*}
        \sum_{n \leq x} \nu_H(n) \chi(n) &= x^s \sum_{n \leq x}\frac{\nu_H(n) \chi(n)}{n^s} - s\int_2^x t^{s-1}\sum_{n \leq t}\frac{\nu_H(n) \chi(n)}{n^s}\:dt\\
        &\ll x^{1-\frac{1}{(\log x)^{1/2 + \epsilon}}}\\
        &\ll \frac{x}{(\log x)^A},
    \end{align*}
    for any $A>0$, as the partial sums in the above equation are bounded when ${\Re(s) = 1-1/(\log y)^{1/2}}$. By applying the orthogonality relations for Dirichlet characters we have, for $(a,q) = 1$,
    \begin{equation}
    \label{starsuppchantal}
        \begin{split}
            \sum_{\substack{n \leq x\\ n\equiv a \bmod q}}\nu_{H}(n) &= \frac{1}{\phi(q)}\sum_{\chi \bmod q}\chi^{-1}(a)\sum_{n \leq x} \chi(n) \nu_{H}(n)\\
        &= \frac{1}{\phi(q)}\sum_{\substack{n \leq x \\ (n,q) = 1}} \nu_{H}(n) +O\left( \frac{x}{(\log x)^{A}}\right).
        \end{split}
    \end{equation}
    This proves that \eqref{nu fn ap eq} holds for all moduli $q$ satisfying $(q,2D) =1$. At this stage, one can apply the classical result of \cite{wirsing1961asymptotische} for sums of bounded multiplicative functions:
    \begin{theorem}[Wirsing]
        Given a non-negative multiplicative function $f(n)$, assume there exists constants $\alpha,\beta$ with $\beta<2$ such that $f(p^k) \leq \alpha \beta^k$ for each prime $p$ and integer $k \geq 2$. Assume further that as $x \to \infty$, one has
        $$\sum_{p \leq x}f(p) \sim \kappa \frac{x}{\log x},$$
        where $\kappa$ is a constant. Under these assumptions, as $x \to \infty$, one has
        $$\sum_{n \leq x}f(n) \sim \frac{e^{\gamma \kappa}}{\Gamma(\kappa)} \frac{x}{\log x}\prod_{p \leq x} \left( \sum_{k \geq 0} \frac{f(p^k)}{p^k} \right),$$
        where $\gamma$ is the Euler-Mascheroni constant.
    \end{theorem}
    In Lemma 1 of \cite{golubeva2001exceptional}, Wirsing's theorem is applied to $\nu_H(n)$, and one finds that $\kappa = \frac{|H|}{2 h},$ and
    \begin{equation}
    \label{starsupp2}
        \begin{split}
            \sum_{n \leq x} \nu_H(n) &= \frac{e^{\gamma \kappa}}{\Gamma(\kappa)} \frac{x}{\log x}\prod_{p \leq x} \left( 1 + \frac{\nu_H(p)}{p} \right)(1+o(1))\\ &= A_3 \frac{x}{\log x} \exp \left( \sum_{p \leq x} \frac{\nu_H(p)}{p} \right)(1+o(1))\\
    &= A_4 \frac{x}{(\log x)^{1-|H|/(2h)}} (1+o(1)),
        \end{split}
    \end{equation}
    for some constants $A_3, A_4$, which can explicitly be written:
    $$A_3 := \frac{e^{\gamma \kappa}}{\Gamma(\kappa)} \exp\left(-\sum_{p}\sum_{k \geq 2} \frac{(-1)^k(\nu_H(p))^k}{kp^k}\right),$$
    $$A_4 := A_3\lim_{N \to \infty} \exp \left( \sum_{p \leq N} \frac{\nu_H(p)}{p} - \frac{|H|}{2h}\log\log N \right).$$
    
    To continue the proof of \Cref{gollemma1ap}, we apply Wirsing's Theorem to the multiplicative function $\nu_H * \chi_0$, where $\chi_0$ is the principal Dirichlet character modulo $q$. We find, again, that $\kappa = \frac{|H|}{2 h}.$ This leads to
    \begin{equation}
    \label{starstarchantal}
        \begin{split}
            \sum_{\substack{n \leq x \\ (n,q) = 1}} \nu_H(n) &= A_3 \prod_{p|q}\left(1 + \frac{\nu_H(p)}{p}\right)^{-1}\frac{x}{\log x} \exp \left( \sum_{p \leq x} \frac{\nu_H(p)}{p} \right)(1+o(1))\\ &= \prod_{p|q}\left(1 + \frac{\nu_H(p)}{p}\right)^{-1}\sum_{\substack{n \leq x }} \nu_H(n)(1+o(1)).
        \end{split}
    \end{equation}
    We remark that the constant $A_3$ in \eqref{starstarchantal} and \eqref{starsupp2} is the very same, since
    $$\sum_{p \leq x}\nu_{H}(p) \sim \sum_{\substack{p \leq x \\ p \nmid q}}\nu_{H}(p).$$
    We obtained the term $\sum_{n \leq x}\nu_{H}(n)$ in \eqref{starstarchantal} by using the first equality in \eqref{starsupp2}.  We can then write 
    $$S(x,q,a) = \sum_{\substack{p_1 \cdots p_r \leq \sqrt{x} \\ p_i \textit{ repr. by }[f_i]\\p_i \nmid q \\p_i \ne p_j}} \sum_{\substack{m \leq x(p_1 \cdots p_r)^{-1} \\ m \equiv a(p_1 \cdots p_r)^{-1} \bmod q}} \nu_H(m) + \sum_{\substack{m \leq \sqrt{ x}\\(m,q) =1}}\nu_H(m) \sum_{\substack{ \sqrt{ x } \leq p_1 \cdots p_r \leq x/m\\ p_i \textit{ repr. by }[f_i]\\p_i \ne p_j\\ p_1 \cdots p_r \equiv m^{-1}a \bmod q}}1$$
    We note that 
    \begin{align*}
        \sum_{\substack{m \leq \sqrt{ x}\\(m,q) =1}}\nu_H(m) \sum_{\substack{ \sqrt{ x } \leq p_1 \cdots p_r \leq x/m\\ p_i \textit{ repr. by }[f_i]\\p_i \ne p_j\\ p_1 \cdots p_r \equiv m^{-1}a \bmod q}}1 &\ll \sum_{m \leq \sqrt{x}}\nu_H(m) \sum_{\substack{ \sqrt{ x } \leq p_1 \cdots p_r \leq x/m\\p_i \ne p_j}}1\\
        &\ll \frac{x}{\log x}(\log \log x)^{r-1} \sum_{m \leq \sqrt{x}} \frac{\nu_H(m)}{m}\\
        &\ll \frac{x}{(\log x)^{1-|H|/(2h)}}(\log \log x)^{r-1},
    \end{align*}
    using summation by parts and Wirsing's theorem in the last line. We also have, using \eqref{starsuppchantal} and \eqref{starstarchantal},
    
    \begin{align}\hspace{-0.7cm}
    \label{golulem1sumeq}
        \sum_{\substack{p_1 \cdots p_r \leq \sqrt{x} \\ p_i \textit{ repr. by }[f_i]\\p_i \nmid q \\p_i \ne p_j}} \sum_{\substack{m \leq x(p_1 \cdots p_r)^{-1} \\ m \equiv a(p_1 \cdots p_r)^{-1} \bmod q}} \nu_H(m)= \frac{1}{\phi(q)}\prod_{p|q}\left(1 + \frac{\nu_H(p)}{p}\right)^{-1}\sum_{\substack{p_1 \cdots p_r \leq \sqrt{x} \\ p_i \textit{ repr. by }[f_i]\\p_i \nmid q \\p_i \ne p_j}} \sum_{\substack{m \leq x(p_1 \cdots p_r)^{-1}}} \nu_H(m)(1+o(1)),
    \end{align}
    and in \cite{golubeva2001exceptional} it is shown that 
    \begin{equation}
    \label{golu final eq}
        \sum_{\substack{p_1 \cdots p_r \leq \sqrt{x} \\ p_i \textit{ repr. by }[f_i]\\ p_i \nmid q\\p_i \ne p_j}} \sum_{\substack{m \leq x(p_1 \cdots p_r)^{-1}}} \nu_H(m) = \frac{A_4}{h^r}\frac{X}{(\log x)^{1-H/(2h)}} (\log \log x)^r(1+o(1)).\footnote{Actually, Golubeva proves this without the condition ``$p_i \nmid q$'', but the argument is the same in any case, and this condition does not have an effect on the main term.}
    \end{equation}
    This proves \Cref{gollemma1ap}.
\end{proof}
Having proven the validity of \eqref{golulemmaapeq}, we may sum it over the disctinct tuples of $[f_1],\cdots [f_r]$ as discussed prior to \Cref{gollemma1ap}. In doing so, we obtain the following result:
\begin{lemma}
    \label{goluthmsfap}
    Let $N'_f(x,q,a)$ note the number of integers not exceeding $x$ which are square-free, coprime to $2D$, and exceptional for the form $f$. One has, for $a,q,\nu_H$ as in \Cref{gollemma1ap}:
    $$N'_f(x,q,a) = \frac{1}{\phi(q)}\prod_{p|q}\left(1+\frac{\nu_H(p)}{p}\right)^{-1}N'_f(x,1,1) (1+o(1)).$$
\end{lemma}
\Cref{goluthmsfap} is a statement about squarefree integers coprime to $2D$, and we also note that 
\begin{equation}
    \label{frf sf}
    B'_f(x,q,a) = B'_R(x,q,a) - N'_f(x,q,a),
\end{equation}
where the $'$ symbol always indicates a sum over squarefree integers which are coprime to $2D$. By our work in this section, we now have an explicit form for the main term of $N'_f(x,q,a)$. This yields the secondary term in the expansion of \Cref{thm3}.  

\subsubsection{An Explicit Numerical Example}
To end \Cref{pf gol thm ap sect}, let us give an explicit example which demonstrates the computational aspects of \Cref{thm3}. Let us fix $f= 2x^2+xy+3y^2$, so that and $q = 3$ so that $D = -23, f \neq f_0,
    \legendre{D}{q} = 1,
    \nu_H(q) = 0, $ and
    $${C(D) = \{x^2+xy+6y^2,2x^2+xy+3y^2,2x^2-xy+3y^2\}},$$
    so that $h=3$, $r = 0$ and $H$ is the trivial subgroup, consisting of the class $f_0 = x^2+xy+6y^2$.

We must then compute the constants $A_3$ and $A_4$ from \eqref{starsupp2}. Using Wolfram Mathematica, we can get a suitable numerical approximation;
$$A_3=\frac{e^{\gamma /6}}{\Gamma(1/6)} \exp \left(-\sum _{k=2}^{\infty}\; \sum _{\substack{p \textit{ repr.}\\ \textit{ by }f_0}} \frac{(-1)^k}{k p^k}\right) = 0.162977...$$
$$A_4 = A_3 \lim_{N \to \infty} \exp \left(\sum _{\substack{p \leq N \\ p \textit{ repr. }f}} \frac{1}{p}-\frac{1}{6} \log \left(\log \left(N\right)\right)\right) = 0.133413...$$
In our computation, we let $k$ go up to $15$, and $p$ go up to $10^8$. Those interested in numerical computations should note that large values of $k$ hardly contribute to the constants, and that the alternating nature of the series for $A_3$ ensures a rapid convergence. In this example, $r=0$, so the constant $A_2(f)$ from \Cref{thm3} is already computed; $A_2(f) = A_4.$
Then, according to \Cref{thm3}, we have
\begin{equation}
\label{computing expll d=-23 r=0}
    N_f'(x,3,a) =\begin{cases}
    \frac{A_4}{2} \frac{x}{(\log x)^{5/6}}(1+o(1)),\; \textit{if }a \equiv 1,2  \bmod 3\\
    o\left(\frac{x}{(\log x)^{5/6}}\right),\; \textit{if }a \equiv 0 \bmod 3\\
\end{cases}
\end{equation}
The following table shows our theorem's main term compared to the actual numerical data for $N_f'(10^8,3,a)$:
\begin{table}[h]
    \centering
    \begin{tabular}{|c|c|c|}
\hline
a   & $N_f'(10^8,3,a)$ & Estimate from \eqref{computing expll d=-23 r=0} \\ \hline
$0$ & $0$              & $o\left(\frac{x}{(\log x)^{5/6}}\right)$           \\
$1$ & $568 737$         & $588 499$          \\
$2$ & $568 775$         & $588 499$          \\ \hline
\end{tabular}
    \caption{Comparing \cref{thm3} with numerical data when $f=2x^2+xy+3y^2$}
    \label{comparison golu 1}
\end{table}
\begin{remark*}
    We see from \cref{comparison golu 1} that $N_f'(10^8,3,0) = 0$. This is no surprise: The integers we are counting are squarefree and of the form $n = m$, where any prime divisor of $m$ must be represented by the principal form $f_0=x^2+xy+6y^2$. Since $q=3$ is not represented by the principal form, then none of the exceptional integers for $f=2x^2+xy+3y^2$ will be divisible by $3$.
\end{remark*}

Though in this example our theorem matches the data well, this is not always the case. It seems that the fit is only good in cases when $r=0$. Indeed, when $r > 0$, the main term of \eqref{golu final eq} has a factor $(\log \log x)^r/h^r$, which is initially very small; it is less than $1$ until $x \geq e^{e^h}$, implying a severe numerical underestimate of our theorem when $x$ is, say, $10^8$.
 \subsection{Dropping the Squarefreeness Condition}
 \label{sectremovesf}
 We have proven an asymptotic count for the number of integers not exceeding $x$ which are squarefree, coprime to $2D$, congruent to $a \bmod q$, and exceptional for a given form $f$. We now comment on the claim made (without proof) by Golubeva that ommiting the condition ``squarefree" changes our asymptotic count only by a constant.
\label{squarefree section}Let $f$ be a fixed reduced form which we decide upon from the start. Let us write $E_f(n)$ to be the function that is $1$ when $n$ is exceptional for $f$, and $0$ otherwise. \begin{lemma}
\label{lemma sf 1e}
    Given a squarefree number $m$ for which $1_R(m) = 1$, and any integer $s$, one has $${E_f(ms^2) = 1 \implies E_f(m) = 1}.$$
\end{lemma} 
\begin{proof}
    We prove the contra-positive statement. Suppose that $E_f(m) = 0$, then $m$ is represented by the form $f$. Since $s^2$ is represented by the principal form (among other forms), then $ms^2$ is also represented by $f$, i.e. $E_f(ms^2) = 0$.
\end{proof}
The converse statement is not true, as we could have some $s$ such that $E_f(m) = 1$, but $E_f(ms^2) = 0$. For example, consider the quadratic form $f(x,y) = x^2+xy+15y^2$. In this case, the integer $5$ is exceptional for $f$, but $5\cdot3^2$ is actually represented by $f$, so not exceptional. Thus, $E_f(5) = 1$, and $E_f(5\cdot3^2) = 0$. By combining \Cref{lemma sf 1e} with \Cref{goluthm}, we have
\begin{align*}
    C x\frac{(\log \log x)^\beta}{(\log x)^{\alpha}}(1+o(1)) = \sum_{n \leq x} E_f(n) &= \sum_{s = 1}^{\sqrt{x}} \sum_{\substack{ms^2 \leq x\\m \textit{ squarefree}}} E_f(ms^2)\\
    &\leq \sum_{s = 1}^{\infty} \sum_{\substack{m \leq x/s^2 \\m \textit{ squarefree}}} E_f(m)\\
    &=C_1 \sum_{s = 1}^{\infty} \frac{x}{s^2}\frac{(\log \log x/s^2)^\beta}{(\log x/s^2)^{\alpha}}(1+o(1))\\
    &= C_1 \sum_{s = 1}^{\infty} \frac{x}{s^2}\frac{(\log \log x)^\beta}{(\log x)^{\alpha}}\left(1+O\left(\frac{\log s}{\log \log x}\right)\right)\\
    &\leq C_1 \frac{\pi^2}{6}x\frac{(\log \log x)^\beta}{(\log x)^{\alpha}}(1+o(1)),
\end{align*}
with $C, C_1, \alpha, \beta$ above being some constants coming from \Cref{goluthm}. The error term involving $\log s$ can be determined using series expansions. Indeed, one finds that
$$\frac{(\log \log x/s^2)^\beta}{(\log x/s^2)^{\alpha}} = \frac{(\log \log x)^\beta}{(\log x)^{\alpha}}\left(1+O\left(\frac{\log s}{\log \log x}\right)\right).$$
All in all, we've given an upper bound on the number of exceptional integers for $f$ not exceeding $x$, and the bound implies that once we drop the squarefree restriction, we only change our count by a constant.\footnote{Assuming that $\sum_{n \leq x}E_f(n)$ converges in the first place.}~Since we have no characterization of the integers $s$ where $E_f(m) = 1$ and $E_f(ms^2) = 0$ simultaneously, we cannot seem to explicitly compute the involved constant.

\section{Conjecture and Observations}
In this section, we state a conjecture on the behaviour of the integers represented by binary quadratic forms in cases not covered by our theorems. It seems natural to conjecture that the bias we have exhibited here in many cases remains true when we drop any restrictions (i.e. we do not assume anything about the class group, and drop any restrictions on squarefreeness/coprimality with $2D$).
\begin{conj}\label{mainconj}
Let $D$ be any fundamental discriminant. Let $f$ be any reduced binary quadratic form of discriminant $D$. Let $q$ be any prime such that $(q,2D)=1$. The secondary term in the asymptotic expansion of $B_f(x;q,a)$ contains a constant which is larger when $a \equiv 0 \bmod q$, and smaller otherwise. 
\end{conj}
The above conjecture is evidenced by much numerical data (see \Cref{numericaldata}), as well as the fact that there always exists a term of size $x/(\log x)^{3/2}$ ``hidden in the error term'' of $B_f(x;q,a)$ that has a constant satisfying the conditions of the conjecture. This second claim is clear from \Cref{equidgeneralemma} together with \eqref{subtractexep}. Proving the conjecture would likely require one to either show that for a given $f$, $N_f(x,q,a)$ exhibits no bias in terms $\gg x/(\log x)^{3/2}$ (as was the case in \Cref{thm2}), or to show that $N_f(x,q,a)$ exhibits a bias in the correct direction (i.e. $N_f(x,q,a)$ is \textit{smaller} when $a \equiv 0 \bmod q$, and \textit{bigger} otherwise), as was the case in \Cref{thm3}.

There seems to be a couple other interesting phenomena which are tangentially related to \Cref{mainconj}. For example, it seems that there is always an accentuation of the bias toward the zero residue class when one jumps from only considering squarefree integers to considering all integers represented by a given form. One can see an example of this by comparing \Cref{tablethm3sf} with \Cref{tablethm3}, or by comparing \Cref{cn3sf} with \Cref{cn3}.
\section{Numerical Data}
\label{numericaldata}
This section contains numerical examples which demonstrate \Cref{thm1,thm2,goluthmsfap}. 

\textbf{Summary of Tables:}
\begin{itemize}
    \item \Cref{cn1ed,cn1nd} display values of $B_f(x;q,a)$ for $f(x,y) = x^2 + xy +y^2$. In this case, there is a single reduced form of the chosen discriminant, making it the simplest case covered by our theorems (i.e. we compare with \Cref{thm1}).
    \item \Cref{cn2ed,cn2nd} display values of $B_f(x;q,a)$ for $f(x,y) = x^2 + 5y^2$. In this case, there are two separate reduced forms\footnote{ The second reduced form of discriminant $-20$ is $2x^2+2xy+3y^2$. A table of data for this form would look almost identical to \Cref{cn2ed} or \Cref{cn2nd}, depending on the value of $(-20|q)$, since the constants in \Cref{equidgeneralemma} do not depend on the genus.} of discriminant $-20$ lying in two separate genera. These two tables illustrate that the phenomena of \Cref{cn1ed,cn1nd} carry over to the situation where there is more than one genus, provided that each genus contains only a single form (\Cref{thm2}).
    \item \Cref{tablethm3sf} provides an example of the behaviour predicted by \Cref{goluthmsfap}. 
    \item \Cref{tablethm3} suggests that the bias towards the zero residue class is still present when we count all integers, instead of just squarefree integers coprime to $2D$. 
    \item \Cref{cn3sf,cn3}, demonstrate that a bias towards the zero class still exists in a case where $\nu_H(q) = 0$. 
    \item \Cref{table conj sup1,table conj sup 2} give further evidence for the truth of \Cref{mainconj} in cases not covered by our theorems.
\end{itemize}
\begin{table}[h]
\centering
\begin{tabular}{cc|c|c|c|ccc}
\cline{1-5} \cline{8-8}
\multicolumn{1}{|c}{$q$} & $a$ & $B_f(10^8,7,a)$ & Main Term   & Two Terms   &&&\multicolumn{1}{|c|}{Additional Information}\\ \cline{1-5} \cline{8-8}
  \multicolumn{1}{|c}{}  & $0$ & $2 342 596$     & $2 126 610$ & $2 305 520$ &&&\multicolumn{1}{|c|}{$f(x,y) = x^2+xy+y^2,$} \\ 
   \multicolumn{1}{|c}{} & $1$ & $2 181 168$     &             & $2 174 480$ &&& \multicolumn{1}{|c|}{$D=-3,$} \\
  \multicolumn{1}{|c}{}  & $2$ & $2 181 169$     &             &             &&& \multicolumn{1}{|c|}{$\legendre{D}{q}=1,$} \\
\multicolumn{1}{|c}{$7$} & $3$ & $2 181 008$     &             &             &&& \multicolumn{1}{|c|}{$C(D)=\textit{trivial},$} \\
  \multicolumn{1}{|c}{}  & $4$ & $2 181 101$     &             &             &&& \multicolumn{1}{|c|}{$G(D)=\textit{trivial},$} \\
  \multicolumn{1}{|c}{}  & $5$ & $2 181 032$     &             &             &&&\multicolumn{1}{|c|}{$B_f(10^8)/q = 2 204 167.$} \\
  \cline{8-8}
   \multicolumn{1}{|c}{} & $6$ & $2 181 096$     &             &            &&& \\
    \cline{1-5} 
\end{tabular}
\caption{The distribution of $B_f(x;7,a)$ for $f(x,y)= x^2+xy+y^2$ compared with the first term and first two terms of \Cref{thm2}. Notice the equidistribution among all residue classes, with a bias towards zero.}
\label{cn1ed}
\end{table}
\vspace{-1cm}
\begin{table}[h]
\centering
\begin{tabular}{ccccccc|c|}
\cline{1-5} \cline{8-8}
\multicolumn{1}{|c}{$q$} & \multicolumn{1}{c|}{$a$} &  \multicolumn{1}{c|}{$B_f(10^8,5,a)$} & Main Term &  \multicolumn{1}{|c|}{Two Terms}  &&&Additional Information\\ \cline{1-5} \cline{8-8}
       \multicolumn{1}{|c}{}   & \multicolumn{1}{c|}{$0$} &   \multicolumn{1}{c|}{$685734$ }    & $595 452$ &  \multicolumn{1}{|c|}{$666 121$} &&&$f(x,y) = x^2+xy+y^2$,\\ 
   \multicolumn{1}{|c}{}  & \multicolumn{1}{c|}{$1$}&     \multicolumn{1}{c|}{$3685946$ }  & $3572710$ &  \multicolumn{1}{|c|}{$3696540$} &&& $D=-3$,\\
   \multicolumn{1}{|c}{$5$}   & \multicolumn{1}{c|}{$2$} &      \multicolumn{1}{c|}{$3685770$}  &           &         \multicolumn{1}{|c|}{}  &&& $\legendre{D}{q}=-1$,\\
 \multicolumn{1}{|c}{}& \multicolumn{1}{c|}{$3$} &    \multicolumn{1}{c|}{$3685731$ }   &           &        \multicolumn{1}{|c|}{} &&& $C(D)=\textit{trivial}$,\\
   \multicolumn{1}{|c}{}  & \multicolumn{1}{c|}{$4$} &     \multicolumn{1}{c|}{$3685990$}   &           &          \multicolumn{1}{|c|}{}  &&& $G(D)=\textit{trivial}$,\\
  \cline{1-5} 
    &     &             &            && &&$B_f(10^8)/q^2 = 617 167$,\\
    &&&&&&& $(q+1)B_f(10^8)/q^2 = 3 703 001$\\
    \cline{8-8}
\end{tabular}
\caption{The distribution of $B_f(x;5,a)$ for $f(x,y)= x^2+xy+y^2$ compared with the first term and first two terms of \Cref{thm2}. Notice the equidistribution among all nonzero residue classes, with a much smaller proportion for the zero residue class.}
\label{cn1nd}
\end{table}
\vspace{-1cm}
\begin{table}[H]
\centering
\begin{tabular}{ccccccc|c|}
\cline{1-5} \cline{8-8}
\multicolumn{1}{|c}{$q$} & \multicolumn{1}{c|}{$a$} & $B_f(10^8,3,a)$ & Main Term   & \multicolumn{1}{c|}{Two Terms}   &&&Additional Information\\ \cline{1-5} \cline{8-8}
   \multicolumn{1}{|c}{} & \multicolumn{1}{c|}{$0$} & $4502885$     & $4156480$ & \multicolumn{1}{c|}{$4448270$} &&&$f(x,y) = x^2+5y^2$,\\ 
 \multicolumn{1}{|c}{$3$}   & \multicolumn{1}{c|}{$1$} & $4276237$     &             & \multicolumn{1}{c|}{$4262350$} &&& $D=-20$,\\
    \multicolumn{1}{|c}{}& \multicolumn{1}{c|}{$2$ }& $4275772$     &             &      \multicolumn{1}{c|}{}       &&& $\legendre{D}{q}=1$,\\
    \cline{1-5}
 &  &     &             &             &&& $C(D) \cong \Z/2\Z$,\\
    &  &     &             &             &&& $G(D)\cong\Z/2\Z$,\\
    &  &      &             &             &&&$B_f(10^8)/q = 4351630$.\\
    \cline{8-8}
\end{tabular}
\caption{The distribution of $B_f(x;3,a)$ for $f(x,y)= x^2+5y^2$ compared with the first term and first two terms of \Cref{thm2}.}
\label{cn2ed}
\end{table}
\begin{table}[H]
\centering
\begin{tabular}{|cc|c|c|c|ccc}
\cline{1-5}\cline{8-8}
$q$  & $a$  & $B_f(10^8,11,a)$ & Main Term & Two Terms &&&\multicolumn{1}{|c|}{Additional Information}\\ \cline{1-5} \cline{8-8}
     & $0$  & $128016$         & $103053$  & $120629$  &&& \multicolumn{1}{|c|}{$f(x,y) = x^2+5y^2,$} \\
     & $1$  & $1292745$        & $1236640$ & $1270480$ &&& \multicolumn{1}{|c|}{$D=-20,$} \\
     & $2$  & $1292628$        &           &           &&& \multicolumn{1}{|c|}{$\legendre{D}{q}=-1,$} \\
     & $3$  & $1292788$        &           &           &&& \multicolumn{1}{|c|}{$C(D)\cong \Z/2\Z,$} \\
     & $4$  & $1292739$        &           &           &&& \multicolumn{1}{|c|}{$G(D)\cong\Z/2\Z$,} \\
$11$ & $5$  & $1292791$        &           &           &&&\multicolumn{1}{|c|}{$B_f(10^8)/q^2 = 107892.$} \\

     & $6$  & $1292573$        &           &          &&&\multicolumn{1}{|c|}{$(q+1)B_f(10^8)/q^2 = 1294700.$} \\ 
     \cline{8-8}
     & $7$  & $1292545$        &           &           \\
     & $8$  & $1292595$        &           &           \\
     & $9$  & $1292875$        &           &           \\
     & $10$ & $1292599$        &           &           \\ \cline{1-5}
\end{tabular}
\caption{The distribution of $B_f(x;11,a)$ for $f(x,y)= x^2+5y^2$ compared with the first term and first two terms of \Cref{thm2}.}
\label{cn2nd}
\end{table}
\vspace{-3cm}
\begin{table}[H]
\centering
\begin{tabular}{|cc|c|cccc}
\cline{1-3}\cline{6-6}
$q$  & $a$  & $B'_f(10^8,17,a)$  &&&\multicolumn{1}{|c|}{Additional Information}\\ \cline{1-3} \cline{6-6}
     & $0$  & $376649$         &&& \multicolumn{1}{|c|}{$f(x,y) =x^2+xy+15y^2$,} \\
     & $1$  & $354287$         &&& \multicolumn{1}{|c|}{$D=-59,$} \\
     & $2$  & $354196$                 &&& \multicolumn{1}{|c|}{$\legendre{D}{q}=1,$} \\
     & $3$  & $354373$                  &&& \multicolumn{1}{|c|}{$C(D)\cong \Z/3\Z,$} \\
     & $4$  & $354313$                  &&& \multicolumn{1}{|c|}{$G(D)=trivial$,} \\
 & $5$  & $354509$                   &&&\multicolumn{1}{|c|}{$\nu_H(q)=1$.} \\
     & $6$  & $354363$                 &&& \multicolumn{1}{|c|}{$r=1$.}\\\cline{6-6}
     & $7$  & $354453$                \\
$17$  & $8$  & $354278$                 \\
     & $9$  & $354228$                 \\
     & $10$ & $354259$                 \\
      & $11$ & $354418$                 \\ 
       & $12$ & $354329$                 \\ 
        & $13$ & $354263$                 \\ 
         & $14$ & $354347$                  \\ 
          & $15$ & $354402$                \\ 
           & $16$ & $354192$                \\\cline{1-3}
\end{tabular}
\caption{Values of $B'_f(x,3,a)$, where $f(x,y)= x^2+xy+15y^2$. A bias towards the zero class is visible, as predicted by \Cref{thm3}.}
\label{tablethm3sf}
\end{table}

\begin{table}[H]
\centering
\begin{tabular}{|cc|c|cccc}
\cline{1-3}\cline{6-6}
$q$  & $a$  & $B_f(10^8,17,a)$  &&&\multicolumn{1}{|c|}{Additional Information}\\ \cline{1-3} \cline{6-6}
     & $0$  & $782426$         &&& \multicolumn{1}{|c|}{$f(x,y) =x^2+xy+15y^2$,} \\
     & $1$  & $683226$         &&& \multicolumn{1}{|c|}{$D=-59,$} \\
     & $2$  & $683405$                 &&& \multicolumn{1}{|c|}{$\legendre{D}{q}=1,$} \\
     & $3$  & $683040$                  &&& \multicolumn{1}{|c|}{$C(D)\cong \Z/3\Z,$} \\
     & $4$  & $683379$                  &&& \multicolumn{1}{|c|}{$G(D)=trivial$,} \\
 & $5$  & $683240$                   &&&\multicolumn{1}{|c|}{$\nu_H(q)=1$.} \\
     & $6$  & $683199$                 &&&  \multicolumn{1}{|c|}{$r=1$.}\\\cline{6-6}
     & $7$  & $683179$                \\
  $17$   & $8$  & $683380$                 \\
     & $9$  & $683427$                 \\
     & $10$ & $683042$                 \\
      & $11$ & $683073$                 \\ 
       & $12$ & $683018$                 \\ 
        & $13$ & $683403$                 \\ 
         & $14$ & $683214$                  \\ 
          & $15$ & $683499$                \\ 
           & $16$ & $683323$                \\\cline{1-3}
\end{tabular}
\caption{Values of $B_f(x;17,a)$, where $f(x,y)= x^2+xy+15y^2$. A bias still exists numerically when we drop the squarefreeness restriction. The bias seems to become more pronounced in doing so (compare with the bias in \Cref{tablethm3sf}).}
\label{tablethm3}
\end{table}

\begin{table}[H]
\centering
\begin{tabular}{cccc|c|}
\cline{1-3} \cline{5-5}
\multicolumn{1}{|c}{$q$} & \multicolumn{1}{c|}{$a$} &\multicolumn{1}{c|}{$B'_f(10^8,3,a)$}   & &Additional Information\\\cline{1-3} \cline{5-5}
   \multicolumn{1}{|c}{} & \multicolumn{1}{c|}{$0$} &   \multicolumn{1}{c|}{$2418331$}& &$f(x,y) = x^2+xy+6y^2$,\\ 
 \multicolumn{1}{|c}{$3$}   & \multicolumn{1}{c|}{$1$} &     \multicolumn{1}{c|}{$2325663$}& &$D=-23$,\\
    \multicolumn{1}{|c}{}& \multicolumn{1}{c|}{$2$}&        \multicolumn{1}{c|}{$2326169$}      & & $\legendre{D}{q}=1$,\\
    \cline{1-3}
 &  &     &         &  $C(D) \cong \Z/3\Z$,\\
    &  &     &           &   $G(D)=trivial$,\\
    &  &     &           &   $\nu_H(q) = 0$.\\
    
    &  &     &           &   $r= 1$.\\
    \cline{5-5}
    
\end{tabular}
\caption{Values of $B'_f(x,3,a)$, where $f(x,y)= x^2+xy+6y^2$.}
\label{cn3sf}
\end{table}

\begin{table}[H]
\centering
\begin{tabular}{cccc|c|}
\cline{1-3} \cline{5-5}
\multicolumn{1}{|c}{$q$} & \multicolumn{1}{c|}{$a$} &\multicolumn{1}{c|}{$B_f(10^8,3,a)$}   & &Additional Information\\\cline{1-3} \cline{5-5}
   \multicolumn{1}{|c}{} & \multicolumn{1}{c|}{$0$} &   \multicolumn{1}{c|}{$6223402$}& &$f(x,y) = x^2+xy+6y^2$,\\ 
 \multicolumn{1}{|c}{$3$}   & \multicolumn{1}{c|}{$1$} &     \multicolumn{1}{c|}{$4240799$}& &$D=-23$,\\
    \multicolumn{1}{|c}{}& \multicolumn{1}{c|}{$2$}&        \multicolumn{1}{c|}{$4239968$}      & & $\legendre{D}{q}=1$,\\
    \cline{1-3}
 &  &     &         &  $C(D) \cong \Z/3\Z$,\\
    &  &     &           &   $G(D)=trivial$,\\
    &  &     &           &   $\nu_H(q) = 0$.\\
    
    &  &     &           &   $r= 1$.\\
    \cline{5-5}
    
\end{tabular}
\caption{Values of $B_f(x;3,a)$, where $f(x,y)= x^2+xy+6y^2$. Notice how the bias towards $0$ seems to become more pronounced when we drop the squarefree restriction (compare with \Cref{cn3sf}).}
\label{cn3}
\end{table}

\begin{table}[H]
\centering
\begin{tabular}{|cc|c|cccc}
\cline{1-3}\cline{6-6}
$q$  & $a$  & $B_f(10^8,7,a)$  &&&\multicolumn{1}{|c|}{Additional Information}\\ \cline{1-3} \cline{6-6}
     & $0$  & $1745576$         &&& \multicolumn{1}{|c|}{$f(x,y) =x^2+xy+22y^2$,} \\
     & $1$  & $1254963$         &&& \multicolumn{1}{|c|}{$D=-87,$} \\
     & $2$  & $1254939$                 &&& \multicolumn{1}{|c|}{$\legendre{D}{q}=1,$} \\
     $7$ & $3$  & $1254519$                  &&& \multicolumn{1}{|c|}{$C(D)\cong \Z/6\Z,$} \\
     & $4$  & $1255006$                  &&& \multicolumn{1}{|c|}{$G(D)\cong \Z/2\Z$,} \\
 & $5$  & $1254481$                   &&&\multicolumn{1}{|c|}{$\nu_H(q)=0$.} \\\cline{6-6}
     & $6$  & $1254492$                 &&&\\\cline{1-3}
\end{tabular}
\caption{Values of $B_f(x;7,a)$, where $f(x,y)= x^2+xy+22y^2$. A bias still exists, though none of our theorems cover this case.}
\label{table conj sup1}
\end{table}

\begin{table}[H]
\centering
\begin{tabular}{cccc|c|}
\cline{1-3} \cline{5-5}
\multicolumn{1}{|c}{$q$} & \multicolumn{1}{c|}{$a$} &\multicolumn{1}{c|}{$B_f(10^8,3,a)$}   & &Additional Information\\\cline{1-3} \cline{5-5}
   \multicolumn{1}{|c}{} & \multicolumn{1}{c|}{$0$} &   \multicolumn{1}{c|}{$4246393$}& &$f(x,y) = 3x^2+xy+8y^2$,\\ 
 \multicolumn{1}{|c}{$3$}   & \multicolumn{1}{c|}{$1$} &     \multicolumn{1}{c|}{$3387811$}& &$D=-95$,\\
    \multicolumn{1}{|c}{}& \multicolumn{1}{c|}{$2$}&        \multicolumn{1}{c|}{$3387781$}      & & $\legendre{D}{q}=1$,\\
    \cline{1-3}
 &  &     &         &  $C(D) \cong \Z/8\Z$,\\
    &  &     &           &   $G(D) \cong \Z/2\Z$,\\
    &  &     &           &   $\nu_H(q) = 1$.\\
    \cline{5-5}
    
\end{tabular}
\caption{Values of $B_f(x;3,a)$, where $f(x,y)= 3x^2+xy+8y^2$. A bias still exists, though none of our theorems cover this case.}
\label{table conj sup 2}
\end{table}

\bibliographystyle{alpha}
\bibliography{main}
\end{document}